\documentclass[a4paper,11pt]{amsart}
\usepackage[english]{babel}
\usepackage{amstext,amsfonts,amsthm,indentfirst,graphicx,amssymb,amscd,epsfig}
\usepackage{amsmath}
\usepackage{graphics,wrapfig}
\usepackage[ansinew]{inputenc}    
\usepackage{bm}
\usepackage{mathrsfs,dsfont}
\usepackage{hyperref}
\usepackage{mathrsfs}
\usepackage{enumitem}

\newtheorem{proposition}{Proposition}[section]
\newtheorem{definition}{Definition}[section]
\newtheorem{theorem}{Theorem}[section]
\newtheorem{corollary}{Corollary}[section]
\newtheorem{remark}{Remark}[section]
\newtheorem{lemma}{Lemma}[section]

\newfont{\bbb}{msbm10 scaled\magstephalf}     

\def\R{\mathbb R}

\def\R{\mbox{\bbb R}}
\def\x{\mathbf{x}}
\def\y{\mathbf{y}}

\def\EO{E_\Omega}
\def\FO{F_\Omega}
\def\GO{G_\Omega}
\def\Lo{L_\Omega}
\def\Muo{M_{1\Omega}}
\def\Mdo{M_{2\Omega}}
\def\No{N_\Omega}
\def\n{\mathbf{n}}
\def\p{p}

\def\I{\mathbf{I}}

\def\II{\mathbf{II}}

\def\w{\mathbf{w}}

\def\Omegam{\mathbf{\Omega}}
\def\Omegamb{\bar{\mathbf{\Omega}}}
\def\Omegab{\bar{\Omega}}
\def\Lambdam{\mathbf{\Lambda}}
\def\Lambdamb{\bar{\mathbf{\Lambda}}}

\def\mum{\boldsymbol{\mu}}

\newcommand{\spann}[2]{\left\langle{#1},{#2}\right\rangle}

\newcommand{\laO}{\lambda_{\Omega}}
\newcommand{\KO}{K_{\Omega}}
\newcommand{\HO}{H_{\Omega}}

\title{Some classes of frontals and its representation formulas}

\author{T. A. Medina-Tejeda}

\date{}

\address{Instituto de Ci\^encias Matem\'aticas e de Computa\c{c}\~ao - Universidade de S\~ao Paulo,
Av. Trabalhador s\~ao-carlense, 400 - Centro,
CEP: 13566-590 - S\~ao Carlos - SP, Brazil}

\email{talexanmedinat@gmail.com}

\thanks{The author was partially supported by CAPES Grant no. PROEX-10359340/D}

\begin{document}
\begin{abstract}
We characterize the extendibility of the normal curvature on frontals and we give a representation formula of this type of frontals. Also we give representation formulas for wavefronts on all types of singularities and other sub classes of these. Some applications to asymptotic curves and lines of curvature on frontals are made.  
\end{abstract}

\maketitle

\section{Introduction}
Frontals are a class of surfaces with singularities in which the study of the differential geometry presents some difficulties due the presence of singularities. One example of this is the study of the normal curvature and as a consequence the asymptotic curves and the lines of curvature through singularities. Another difficulty is the construction of frontals with some desired geometrical properties due to the shortage of formulas in the literature. Most of the existing work began and is focused in frontals with generic singularities (see \cite{arnld-sing-caus,ishifrontal2,front} for example) but frontals with some desired properties do not always have these good types of singularities. Wavefronts with vanishing mean curvature for instance have only singularities of rank 0 (see \cite{med2} for details). We are going to proceed treating as much as possible all kinds of singularities. 

We introduce the {\it relative normal curvature} in section 3, a function that is well defined even on singularities and allows us to indirectly study the normal curvature, asymptotic curves and lines of curvature. We characterize the extendibility of the normal curvature (theorem \ref{tp1}) in terms of an order relation that we introduced in section 3 and which will be useful in further works. We construct an explicit representation formula for frontals with extendable normal curvature near singularities of rank 1 (theorem \ref{repn}) and also in the general case (proposition \ref{tmb}), but with a condition involving a partial differential equation. These frontals result with extendable Gaussian and mean curvature simultaneously, which is very unusual. 

We mention the works \cite{Mart, murataumehara} in which were obtained representation formulas for wavefronts with prescribed unbounded mean curvature and developable frontals respectively. In section 4, we give representation formulas for wavefronts near all types of singularities separately (theorems \ref{lf1} and \ref{fr0} ). We use these to obtain other formulas for wavefronts with extendable Gaussian curvature (corollary \ref{egc}) and parallelly smoothable (corollaries \ref{corp1} and \ref{corp0}). By last in section 5, we apply some of the results to get the structure of asymptotic curves (theorem \ref{asyt1} and \ref{asyt2}) and lines of curvature (theorem \ref{lct1}) on frontals with extendable normal curvature and wavefronts with extendable negative Gaussian curvature. 

\section{Fixing notation, definitions and some basic results}\label{section-definition}
In this paper, all the maps and functions are of class $C^\infty$. We denote $U$,$V$ and sometimes with subscript added as open sets in $\R^2$, when is not mentioned anything about them. Let $\x:U \to \R^3$ be a smooth map, we call a {\it tangent moving basis} (tmb) of $\x$ a smooth map $\Omegam:U \to \mathcal{M}_{3\times2}(\R)$ in which the columns $\w_1, \w_2:U \to \R^3$ of the matrix $\Omegam=\begin{pmatrix}\w_1&\w_2\end{pmatrix}$ are linearly independent smooth vector fields and  $\x_u,\x_v \in P_\Omega:=\spann{\w_1}{\w_2}$, where $\spann{}{}$ denotes the linear span vector space. 

A smooth map $\x: U \to \R^3$ defined in an open set $U \subset \R^2$ is called a {\it frontal} if, for all $\p\in U$ there exists a unit normal vector field $\n_p:V_p \to \R^3$ along $\x$ (i.e $\x_u$, $\x_v$ are orthogonal to $\n$), where $V_p$ is an open set of $U$, $\p\in V_p$. If the {\it singular set} $\Sigma(\x)=\{\p\in U: \x \text{ is not immersive at $\p$} \}$ has empty interior we call $\x$ a {\it proper frontal} and if $(\x,\n_p):U \to \R^3\times S^2$ is an immersion for all $\p\in U$ we call $\x$ a {\it wavefront} or simply {\it front}. It is known that a smooth map $\x:U \to \R^3$ is a frontal if and only if there exist tangent moving bases of $\x$ locally. Since we are interested in exploring local properties of frontals, we always assume that we have a global tmb $\Omegam$ for $\x$. We denote by $\n:=\frac{\w_1\times\w_2}{\|\w_1\times\w_2\|}$ the normal vector field induced by $\Omegam$. 

We write $\mathbf{f}:(U,p) \to (\R^n,q)$ a map germ, where $\mathbf{f}(p)=q$. We say that $\mathbf{f}_1:(U_1,p)\to (\R^n,q)$ is $\mathscr{R}$-{\itshape equivalent} to $\mathbf{f}_2:(U_2,p)\to (\R^n,q)$, if there exist a diffeomorphism $\mathbf{h}:(U_2,p)\to (U_1,p)$ such that $\mathbf{f}_2=\mathbf{f}_1\circ\mathbf{h}$. We say that $\mathbf{f}_1$ is $\mathscr{A}$-{\itshape equivalent} to $\mathbf{f}_2$, if there exist $\mathbf{h}$ as before and a diffeomorphism $\mathbf{k}:(\R^n,q)\to (\R^n,q)$ such that $\mathbf{f}_2=\mathbf{k}\circ\mathbf{f}_1\circ\mathbf{h}$. We denote by $D\mathbf{f}:=(\frac{\partial \mathbf{f}_i}{\partial x_j})$, the Jacobian matrix of $\mathbf{f}$ and we consider it as a smooth map $D\mathbf{f}:U \to \mathcal{M}_{n\times2}(\R)$. We write $D\mathbf{f}_{x_1}$, $D\mathbf{f}_{x_2}$ the partial derivatives of $D\mathbf{f}$ and $D\mathbf{f}(\p):=(\frac{\partial \mathbf{f}_i}{\partial x_j}(\p))$ for $\p \in U$. Also, vectors in $\R^n$ are identified as column vectors in $\mathcal{M}_{n\times1}(\R)$ and if $\mathbf{A}\in \mathcal{M}_{n\times n}(\R)$, $\mathbf{A}_{(i)}$ is the $i^{th}$-row and $\mathbf{A}^{(j)}$ is the $j^{th}$-column of $\mathbf{A}$. The trace and adjoint of a matrix are denoted by $tr()$ and $adj()$ respectively. The identity matrix is denoted by $id$. 

Let $\x:U \to \R^3$ be a frontal, $\Omegam$ a tmb of $\x$. Denoting $()^T$ the operation of transposing a matrix, we set the matrices of the first and second fundamental forms: 
$$\mathbf{I}=
\begin{pmatrix}
E&F\\
F&G
\end{pmatrix}:=D\x^TD\x,\ \mathbf{II}=\begin{pmatrix}
L&M\\
M&N
\end{pmatrix}:=-D\x^T D\n.$$	
The Weingarten matrix $\boldsymbol{\alpha}:=-\II^{T}\I^{-1}$ is defined in $\Sigma(\x)^c$. Also, we set the matrices:
$$\I_\Omega=\begin{pmatrix}
\EO & \FO\\
\FO & \GO\end{pmatrix}:=\Omegam^T\Omegam\label{IO},\ \II_\Omega=\begin{pmatrix}
\Lo&\Muo\\
\Mdo&\No
\end{pmatrix}:=-\Omegam^T D\n\label{IIO},$$
$$\boldsymbol{\mu}_{\Omega}:=-\II_\Omega^T\I_\Omega^{-1}\label{W},\ \Lambdam_{\Omega}:=D\x^T\Omegam(\I_\Omega)^{-1},\ \boldsymbol{\alpha}_\Omega:=\boldsymbol{\mu}_{\Omega}adj(\Lambdam_\Omega).$$
If $\x$ is a frontal and $\Omegam$ a tmb of $\x$, we write simply $\Lambdam=(\lambda_{ij})$ and $\mum=(\mu_{ij})$ instead of $\Lambdam_\Omega$ and $\boldsymbol{\mu}_{\Omega}$ when there is no risk of confusion, $\laO:=det(\Lambdam)$ and $\mathfrak{T}_\Omega(U)$ as the principal ideal generated by $\laO$ in the ring $C^\infty(U,\R)$. The matrix $\Lambdam$ and $\boldsymbol{\mu}$ satisfy $D\x=\Omegam\Lambdam^T$ and $D\n=\Omegam\mum^T$ (see \cite{med}), thus $\Sigma(\x)=\laO^{-1}(0)$ and $rank(D\x)=rank(\Lambdam)$.

The following propositions were proved in \cite{med} and we are going to use them frequently. These are about the relative curvature which we define as follows.
\begin{definition}
	Let $\x:U \to \R^3$ be a frontal, $\Omegam$ a tmb of $\x$, the {\it $\Omega$-relative curvature} and the {\it $\Omega$-relative mean curvature} are defined on $U$ by $\KO=det(\boldsymbol{\mu}_{\Omega})$ and $H_\Omega=-\frac{1}{2}tr(\boldsymbol{\alpha}_\Omega)$ respectively.
\end{definition}

\begin{proposition}\cite[Proposition 3.17]{med}\label{lim} 
	Let $\x:U \to \R^3$ be a proper frontal, $\Omegam$ a tangent moving basis  of $\x$, $\KO$, $H_\Omega$, $K$ and $H$ the $\Omega$-relative curvature, the $\Omega$-relative mean curvature, the Gaussian curvature and the mean curvature of $\x$ respectively.Then,	
	\begin{enumerate}[label=(\roman*)]
		\item for $\p\in \Sigma(\x)^c$, $\KO=\laO K$ and $H_\Omega=\laO H$,
		\item for $\p\in \Sigma(\x)$, $\KO=\lim\limits_{(u,v)\to p}\laO K$ and $H_\Omega=\lim\limits_{(u,v)\to p}\laO H$, 
		
	\end{enumerate} 
	where the right sides are restricted to the open set $\Sigma(\x)^c$.	
\end{proposition}

\begin{proposition}\cite[Theorem 3.22]{med}\label{wft}
	Let $\x:U \to \R^3$ be a frontal, $\Omegam$ a tangent moving basis of $\x$ and  $\p \in \Sigma(\x)$. Then,
	\begin{enumerate}[label=(\roman*)]
		\item $\x:U \to \R^3$ is a front on a neighborhood $V$ of $\p$ with $rank(D\x(\p))=1$ if and only if $H_\Omega(\p)\neq 0$. 
		
		\item $\x:U \to \R^3$ is a front on a neighborhood $V$ of $\p$ with $rank(D\x(\p))=0$ if and only if $H_\Omega(\p)=0$ and $\KO(\p)\neq0$.
		
	\end{enumerate}
\end{proposition}

\section{Representation formula of frontals with extendable normal curvature}

We are going to introduce a relation on the set of frontals defined on the same domain. This relation results useful to characterize the extendibility of the normal curvature as we shall see later.  

\begin{definition}
	Let $\x_1$, $\x_2$ be frontals defined on $U$. We define the relation $\precsim$ on the set of frontals defined on $U$ by $\x_1 \precsim \x_2$ if there exist $\Omegam_1$, $\Omegam_2$ tangent moving bases of $\x_1$ and $\x_2$ respectively and a smooth matrix-valued map $\mathbf{B}:U \to \mathcal{M}_{2\times2}(\R)$ such that $\Lambdam_{\Omega_2}=\Lambdam_{\Omega_1}\mathbf{B}$. Also, we define the equivalence relation $\sim$ by $\x_1 \sim \x_2$ if there exist a smooth matrix-valued map $\mathbf{B}:U \to GL(2,\R)$ such that $\Lambdam_{\Omega_2}=\Lambdam_{\Omega_1}\mathbf{B}$.  
\end{definition}
It easy to verify that $\sim$ is an equivalence relation as also $\precsim$ is reflexive and transitive. If there exists a smooth matrix-valued map $\mathbf{B}:U \to \mathcal{M}_{2\times2}(\R)$ such that $\Lambdam_{\Omega_2}=\Lambdam_{\Omega_1}\mathbf{B}$ and $\bar{\Omegam}_1$ is another tmb of $\x_1$, then $\Omegam_1\Lambdam_{\Omega_1}^T=\bar{\Omegam}_1\Lambdam_{\bar{\Omega}_1}^T$ and hence $\Lambdam_{\Omega_1}=\Lambdam_{\bar{\Omega}_1}\bar{\Omegam}_1^T\Omegam_1\I_{\Omega_1}^{-1}$, which can be substituted in the first equality to obtain that there exists a smooth matrix-valued map $\mathbf{C}:U \to \mathcal{M}_{2\times2}(\R)$, such that $\Lambdam_{\Omega_2}=\Lambdam_{\bar{\Omega}_1}\mathbf{C}$. Analogously, we can change $\Omegam_2$ to another tmb of $\x_2$ and preserve the relation. Thus, $\precsim$ does not depend on the pair of chosen tangent moving bases, even if $\x_1, \x_2$ are not proper frontals. 

Also, if we consider the classes of equivalence by $\sim$ of proper frontals, $\precsim$ induces an order relation on the set of these classes, namely $\precsim$ is antisymmetric. In fact if $\x_1 \precsim \x_2$ and $\x_2 \precsim \x_1$, there exits smooth matrix-valued maps $\mathbf{B}, \mathbf{C}$ such that $\Lambdam_{\Omega_2}=\Lambdam_{\Omega_1}\mathbf{B}$ and $\Lambdam_{\Omega_1}=\Lambdam_{\Omega_2}\mathbf{C}$. Thus $\Lambdam_{\Omega_2}=\Lambdam_{\Omega_2}\mathbf{C}\mathbf{B}$ and therefore $\mathbf{C}\mathbf{B}=id$ on regular points. Since regular point are dense, we get that $\mathbf{B}, \mathbf{C}$ are invertible on $U$, which is by definition $\x_1 \sim \x_2$. For this reason, we name $\precsim$ by {\it $\Lambda$-order} and $\sim$ by {\it $\Lambda$-equivalence}.  

\begin{definition}
	Let $\x:U \to \R^3$ be a frontal, $\Omegam$ a tangent moving basis of $\x$, we define the {\it $\Omega$-relative normal curvature} by:
	$$k_p^{\Omega}(\omega):=\frac{\omega^T\II_{\Omega}adj(\Lambdam_{\Omega}^T)\omega}{\omega^T\I_\Omega \omega},$$
	where $\p\in U$ and $\omega\in \R^2-\{0\}$ represent the coordinates in the basis $\Omegam$ of vectors in the plane $P_\Omega$.
\end{definition}
We remember that the classical normal curvature on a regular point $p\in U$ is given by:
$$k_p(\zeta):=\frac{\zeta^T\II\zeta}{\zeta^T\I\zeta},$$
where $\zeta \in \R^2-\{0\}$ is the coordinate of a vector in the tangent plane in the basis $D\x$. Observe that the coordinate of the same vector in the basis $\Omegam$ is $\omega=\Lambdam_{\Omega}^T\zeta$. 

If $\x$ is a proper frontal, $\Omegam$ and $\hat{\Omegam}$ are tmbs of $\x$ inducing normal vectors with the same (resp. opposite) orientation, then $\hat{\Omegam}=\Omegam\mathbf{B}$ where $\mathbf{B}$ has positive (resp. negative) determinant. It is easy to verify that $k_p^{\Omega}(\omega)=det(\mathbf{B})k_p^{\hat{\Omega}}(\hat{\omega})$, where $\omega$, $\hat{\omega}$ are the coordinates of a vector $\mathbf{v}\in P_{\Omega}=P_{\hat{\Omega}}$ in the bases $\Omegam$, $\hat{\Omegam}$ respectively. Thus, an extreme value of $k_p^{\Omega}(\omega)$ is achieved in $\omega$ if and only if $k_p^{\hat{\Omega}}(\hat{\omega})$ achieves an extreme value in $\hat{\omega}$. The directions defined by the vectors $\mathbf{v}\in P_{\Omega}$ represented by $\omega$ in which $k_p^{\Omega}(\omega)$ achieves a extreme value are called {\it principal directions}. The directions in which $k_p^{\Omega}(\omega)=0$ are called {\it asymptotic directions}. Observe that, these directions does not depend on the chosen tangent moving basis. However, in the case of principal directions, when we change to a tmb with opposite normal vector, maximum changes to minimum and vice verse, but when the normal vector is the same, maximum and minimum are preserved.

\begin{definition}
	Let $\x:U \to \R^3$ be a proper frontal and $\Omegam$ a tangent moving basis of $\x$. We say that the normal curvature has a smooth extension or simply the normal curvature is extendable if the function $k_p(\Lambdam_{\Omega}^{-T}\omega):\Sigma(\x)^c\times (\R^2-\{0\})\to \R$ has a smooth extension to $U\times (\R^2-\{0\})$, where $\omega \in \R^2-\{0\}$ and $\p \in \Sigma(\x)^c$. 
\end{definition}
The following proposition shows why the relative normal curvature can be used to study the classical normal curvature near singularities.
\begin{proposition}\label{plaO}
	Let $\x:U \to \R^3$ be a proper frontal, $\Omegam$ a tangent moving basis of $\x$ and $p \in \Sigma(\x)^c$. Then, $k_p^{\Omega}(\omega)=\laO k_p(\zeta)$, where $\omega$ and $\zeta$ are the coordinates of the same vector in the moving bases $\Omegam$ and $D\x$ respectively.
\end{proposition}
\begin{proof}
	As $\omega=\Lambdam_{\Omega}^T\zeta$ and since $\I=\Lambdam_{\Omega}\I_\Omega\Lambdam_{\Omega}^T$, $\II=\Lambdam_{\Omega}\II_{\Omega}$ then 
	\begin{align*}
		k_p^{\Omega}(\omega)=\frac{\zeta^T\Lambdam_{\Omega}\II_{\Omega}adj(\Lambdam_{\Omega}^T)\Lambdam_{\Omega}^T\zeta}{\zeta^T\Lambdam_{\Omega}\I_\Omega \Lambdam_{\Omega}^T \zeta}=\laO\frac{\zeta^T\II\zeta}{\zeta^T\I\zeta}=\laO k_p(\zeta).
	\end{align*}
\end{proof}
Before proceeding we need the following lemma for some observations and the next theorem. 
\begin{lemma}\label{sym}
	Let $\x:U \to \R^3$ be a proper frontal and $\Omegam$ a tangent moving basis of $\x$. The matrix $\II_{\Omega}adj(\Lambdam_{\Omega}^T)$ is symmetric.
\end{lemma}
\begin{proof}
	As $\II=\Lambdam_{\Omega}\II_{\Omega}$ is symmetric, then $adj(\Lambdam_{\Omega})\II adj(\Lambdam_{\Omega}^T)=\laO \II_{\Omega}adj(\Lambdam_{\Omega}^T)$ as well and by density of regular point in $U$, follows the result.
\end{proof}
Let $\x:U \to \R^3$ be a proper frontal. Observe that, if $\Omegam$ is a tmb of $\x$ with orthonormal columns, $k_p^{\Omega}(\omega)$ restricted to $\omega$ with $|\omega|=1$ is equal to $\omega^T\II_{\Omega}adj(\Lambdam_{\Omega}^T)\omega$. As the matrix $\II_{\Omega}adj(\Lambdam_{\Omega}^T)$ is symmetric, the extreme values of $\omega^T\II_{\Omega}adj(\Lambdam_{\Omega}^T)\omega$ are the eigenvalues of $\II_{\Omega}adj(\Lambdam_{\Omega}^T)$ (see \cite{dc},chapter 3, appendix) which are exactly the relative principal curvatures $k_{1\Omega}=\HO-\sqrt{\HO^2-\laO\KO}$, $k_{2\Omega}=\HO+\sqrt{\HO^2-\laO\KO}$ introduced in \cite{med2}.

\begin{remark}\label{eigenv}
When the columns of $\Omegam$ are orthonormal, a principal direction in $P_{\Omega}$ is represented by an eigenvector $\omega$ of $\II_{\Omega}adj(\Lambdam_{\Omega}^T)$. In this tangent moving base $\II_{\Omega}adj(\Lambdam_{\Omega}^T)=-\mum_{\Omega}^Tadj(\Lambdam_{\Omega}^T)$. In fact, for another arbitrary tangent moving basis $\hat{\Omegam}$, we have that $\omega$ is an eigenvector of $-\mum_{\Omega}^Tadj(\Lambdam_{\Omega}^T)$ if and only if $\hat{\omega}=\mathbf{B}^{-1}\omega$ is an eigenvector of $-\mum_{\hat{\Omega}}^Tadj(\Lambdam_{\hat{\Omega}}^T)$, where $\mathbf{B}$ is a smooth matrix-valued map such that $\hat{\Omegam}=\Omegam\mathbf{B}$. To see this, note that $\I_{\hat{\Omega}}=\mathbf{B}^T\mathbf{B}$, $\Lambdam_{\hat{\Omega}}^T=\mathbf{B}^{-1}\Lambdam_{\Omega}^T$ and  $\II_{\hat{\Omega}}=\pm\mathbf{B}^T\II_{\Omega}$ (the sign depends on whether $\Omegam$ and $\hat{\Omegam}$ induce the same or opposite normal vectors). Thus, there exists a scalar $r(p)$ such that $\II_{\Omega}adj(\Lambdam_{\Omega}^T)\omega=r(p)\omega$ if and only if $det(\mathbf{B}^{-1})\mathbf{B}^T\II_{\Omega}adj(\Lambdam_{\Omega}^T)\mathbf{B}\mathbf{B}^{-1}\omega=det(\mathbf{B}^{-1})r(p)\mathbf{B}^T\mathbf{B}\mathbf{B}^{-1}\omega$ which is the same as $\II_{\hat{\Omega}}adj(\Lambdam_{\hat{\Omega}}^T)\hat{\omega}=det(\mathbf{B}^{-1})r(p)\I_{\hat{\Omega}}\hat{\omega}$. From this follows the assertion.

On the other hand, since the extreme values of $k_p^{\hat{\Omega}}$ have to be $det(\mathbf{B}^{-1})k_{1\Omega}$, $det(\mathbf{B}^{-1})k_{2\Omega}$ and $H_{\hat{\Omega}}=det(\mathbf{B}^{-1})H_{\Omega}$, $K_{\hat{\Omega}}=det(\mathbf{B}^{-1})K_{\Omega}$, using the definition of the relative principal curvature, we can conclude that the extreme values of $k_p^{\hat{\Omega}}$ are $k_{1\hat{\Omega}}$ and $k_{2\hat{\Omega}}$ for an arbitrary tangent moving basis $\hat{\Omegam}$.    
\end{remark}    

Now, we proceed to characterize the extendibility of the normal curvature.

\begin{theorem}\label{tp1}
	Let $\x:U \to \R^3$ be a proper frontal, $\Omegam$ a tangent moving basis of $\x$ and  $\n$ the normal vector field induced by $\Omegam$. Then, the following statements are equivalent:
	\begin{enumerate}[label=(\roman*)]
		\item The normal curvature has a smooth extension. 
		
		\item The entries of  $\II_{\Omega}adj(\Lambdam_{\Omega}^T)$ belong to $\mathfrak{T}_\Omega(U)$. 
		
		\item $\x \precsim \n$.
		
	\end{enumerate}
\end{theorem}
\begin{proof}\
	\begin{itemize}
		\item $(i)\Rightarrow (ii)$ If the normal curvature has a smooth extension then by proposition \ref{plaO} $k_p^{\Omega}(\omega) \in \mathfrak{T}_\Omega(U)$ with $\omega$ fixed. Thus $$k_p^{\Omega}(\omega)\omega^T\I_\Omega \omega=\omega^T\II_{\Omega}adj(\Lambdam_{\Omega}^T)\omega \in \mathfrak{T}_\Omega(U)$$
		and denoting $e_1,  e_2$ the canonical base of $\R^2$, by lemma \ref{sym} we have that $e_i^T\II_{\Omega}adj(\Lambdam_{\Omega}^T)e_j=\frac{1}{2}[(e_i+e_j)^T\II_{\Omega}adj(\Lambdam_{\Omega}^T)(e_i+e_j)-e_i^T\II_{\Omega}adj(\Lambdam_{\Omega}^T)e_i-e_j^T\II_{\Omega}adj(\Lambdam_{\Omega}^T)e_j]$. From this follows the result.   
		\item $(ii)\Rightarrow (iii)$ If the entries of  $\II_{\Omega}adj(\Lambdam_{\Omega}^T)$ belong to $\mathfrak{T}_\Omega(U)$, there exist a matrix-valued map $\mathbf{B}$ such that $\II_{\Omega}adj(\Lambdam_{\Omega}^T)=\mathbf{B}\laO=\mathbf{B}\Lambdam_{\Omega}^Tadj(\Lambdam_{\Omega}^T)$. Then by density of regular points $\II_{\Omega}=\mathbf{B}\Lambdam_{\Omega}^T$ and therefore $\mum_{\Omega}=-\Lambdam_{\Omega}\mathbf{B}^T\I_\Omega^{-1}$. Remembering that $D\n=\Omegam\mum_{\Omega}^T$ we have the result.
		\item $(iii)\Rightarrow (i)$ If there exist a smooth matrix-valued map $\mathbf{C}$ such that $\mum_{\Omega}=\Lambdam_{\Omega}\mathbf{C}$, then we have $\II_{\Omega}=\mathbf{B}\Lambdam_{\Omega}^T$, where $\mathbf{B}=-\I_\Omega  \mathbf{C}^T$. Thus, on regular points $$k_p(\Lambdam_{\Omega}^{-T}\omega)=\frac{\omega^T\II_{\Omega}\Lambdam_{\Omega}^{-T}\omega}{\omega^T\I_\Omega \omega}=\frac{\omega^T\mathbf{B}\omega}{\omega^T\I_\Omega \omega}$$ which is extendable to the entire domain $U$.
	\end{itemize}
\end{proof}
 
 \begin{corollary}\label{lamb}
 	Let $\x:U \to \R^3$ be a proper frontal with extendable normal curvature, then the Gaussian curvature and mean curvature have smooth extensions. Furthermore, this extension of the Gaussian curvature is non-vanishing if and only if $\x \sim \n$.
 \end{corollary}
\begin{proof}
	By theorem \ref{tp1} $\x \precsim \n$, then there exists a smooth matrix-valued map $\mathbf{B}:U \to \mathcal{M}_{2\times2}(\R)$, such that $\mum_{\Omega}=\Lambdam_{\Omega}\mathbf{B}$. Therefore on the regular points, by proposition \ref{lim} $K=\frac{\KO}{\laO}=det(\mathbf{B})$ and $H=\frac{H_\Omega}{\laO}=-\frac{1}{2\laO}tr(\boldsymbol{\mu}_{\Omega}adj(\Lambdam_\Omega))=-\frac{1}{2}tr(\mathbf{B})$ which are extendable to the entire domain $U$.  
\end{proof}

With the purpose of finding examples of frontals with extendable normal curvature, we are going to construct a representation formula explicitly near singularities of rank 1. The following proposition gives us a general representation formula of this type of frontals, but in terms of functions satisfying a compatibility condition, which still makes it difficult to generate good examples. However, near singularities of rank 1, using this proposition we can obtain a formula without involving a compatibility condition in theorem \ref{repn}.
\begin{proposition}\label{tmb}
	Let $\x:U \to \R^3$ be a proper frontal with extendable normal curvature, then after a rigid motion this locally has a tangent moving basis in the following form :
	$$\Omegam=\begin{pmatrix}
		1 & 0\\
		0 & 1\\
		g_1 & g_2\end{pmatrix},\ \Lambdam_{\Omega}^T=D(a,b),$$ satisfying
	\begin{align}
		D(g_1,g_2)=\begin{pmatrix}
			h_1 & h_2\\
			h_2 & h_3\end{pmatrix}D(a,b),
	\end{align}
	where $g_1,g_2,h_1,h_2,h_3$ are smooth functions.
\end{proposition}
\begin{proof}
In \cite{med} was seen that a reduced tangent moving basis like the above one always exists locally. After a rigid motion we can choose one like this. Then, by theorem \ref{tp1} there exists a smooth matrix-valued map $\mathbf{B}$ such that $\II_{\Omega}=\mathbf{B}\Lambdam_{\Omega}^T$. Since $\II_{\Omega}=D(g_1,g_2)(1+g_1^2+g_2^2)^{-\frac{1}{2}}$ and $\Lambdam_{\Omega}^T=D(a,b)$ in this tangent moving basis, where $\x=(a,b,c)$, we get that $D(g_1,g_2)=\mathbf{H}D(a,b)$ with $\mathbf{H}=\mathbf{B}(1+g_1^2+g_2^2)^{\frac{1}{2}}$. By corollary 3.8 in \cite{med} $D(a,b)^TD(g_1,g_2)$ is symmetric, therefore $\mathbf{H}$ too and we have the result.    
\end{proof}

\begin{theorem}\label{repn}
	Let $\x:(U,0) \to (\R^3,0)$ be a proper frontal with extendable normal curvature and $0$ a singularity of rank 1, then after a rigid motion and a change of coordinates on a neighborhood of $0$, $\x$ can be represented by the formula:
	\begin{equation}\label{eqnormal}
		\begin{aligned}
			(u,b(u,v),
			&\int_{0}^{v}\int_{0}^{t_2}h(u,t_1)b_v(u,t_1)dt_1b_v(u,t_2)dt_2+\int_{0}^{v}\int_{0}^{u}l(t_1)dt_1b_v(u,t_2)dt_2\\
			&+\int_{0}^{u}\int_{0}^{t_2}l(t_1)dt_1b_u(t_2,0)dt_2+\int_{0}^{u}\int_{0}^{t_2}r(t_1)dt_1dt_2),
		\end{aligned}
	\end{equation}
where $b,h,l,r$ are smooth function on  neighborhoods of the origin in each case.
\end{theorem}
\begin{proof}
	We can assume $\x$ has the form $(u,b(u,v),c(u,v))$ and we can find a tangent moving basis $\Omegam$ like proposition \ref{tmb}. Thus, we have
	\begin{align}\label{cond1}
		D(g_1,g_2)=\begin{pmatrix}
			h_1 & h_2\\
			h_2 & h\end{pmatrix}\begin{pmatrix}
			1 & 0\\
			b_u & b_v\end{pmatrix},
	\end{align}
then $g_{1v}=h_2b_v$ and $g_{2v}=hb_v$. Thus, $g_{1u}=\int_{0}^{v}(h_2(u,t)b_v(u,t))_udt+r(u)$ and $g_{2u}=\int_{0}^{v}(h(u,t)b_v(u,t))_udt+l(u)$ for some smooth functions $r(u)$ and $l(u)$. We have $\int_{0}^{v}(h_2(u,t)b_v(u,t))_udt+r(u)=h_1+h_2b_u$, $\int_{0}^{v}(h(u,t)b_v(u,t))_udt+l(u)=h_2+hb_u$ therefore $h_2$ and $h_1$ can be determined by $b,h,l,r$ with these relations and consequently $g_{1u},g_{1v},g_{2u},g_{2v}$ too. Integrating we get $$g_{1}(u,v)=\int_{0}^{v}h_2(u,t)b_v(u,t)dt+\int_{0}^{u}r(t)dt,$$ $$g_{2}(u,v)=\int_{0}^{v}h(u,t)b_v(u,t)dt+\int_{0}^{u}l(t)dt$$ and substituting $h_2$ in $g_1$ we obtain
\begin{equation}
\begin{aligned}\nonumber g_1(u,v)=&\int_{0}^{v}\int_{0}^{t_2}(h(u,t_1)b_v(u,t_1))_udt_1b_v(u,t_2)dt_2+\int_{0}^{v}l(u)b_v(u,t)dt\\&-\int_{0}^{v}b_u(u,t)h(u,t)b_v(u,t)dt+\int_{0}^{u}r(t)dt.
\end{aligned}
\end{equation}
Since condition (\ref{cond1}) implies that $D(a,b)^TD(g_1,g_2)$ is symmetric, which is equivalent to $(g_1+b_ug_2)_v=(b_vg_2)_u$, then the system
\begin{align*}
	&c_u=g_1+b_ug_2,\\
	&c_v=b_vg_2\\
	&c(0,0)=0
\end{align*}
has a unique solution locally. Then, $c(u,v)=\int_{0}^{v}b_v(u,t)g_2(u,t)dt+c(u,0)$, but $c_u(u,0)=g_1(u,0)+b_u(u,0)g_2(u,0)$, therefore $$c(u,0)=\int_{0}^{u}\int_{0}^{t_1}r(t)dtdt_1+\int_{0}^{u}b_u(t_1,0)\int_{0}^{t_1}l(t)dtdt_1.$$
Substituting $c(u,0)$ and $g_1(u,v)$ in $c(u,v)$, we get $\x=(u,b,c)$ represented with the formula (\ref{eqnormal}).		     	
\end{proof}
Some frontals with extendable normal curvature present some types of singularities that we call false singularities. Close to these, the image of the frontal looks like a piece of a regular surface. Formula (\ref{eqnormal}) produces some of these, as well as frontals with extendable normal curvature without false singularities like $\x=(u,\frac{2}{5}v^5+v^2,uv^2)$ on $(-1,1)\times (-1,1)$ for instance, where $b=\frac{2}{5}v^5+v^2, h=\frac{-3uv}{2(1+v^3)^3}, l=1, r=0$ are the chosen functions.

A smooth map $x:U\to \R^3$ with $\x(U) \subset S$, where $S$ is a regular surface, can be decomposed locally in the following way. Let $V_1\subset U$ be an open set and $\phi:V_3\to V_2$ a chart of $S$ with $V_3\subset S$, $V_2\subset \R^2$ open sets and $\x(V_1)\subset V_3$, then $\x=\phi^{-1}\circ \phi \circ \x$. Observe that $\phi^{-1}$ is an immersion and $\phi \circ \x$ is a smooth map between open sets of $\R^2$, which has the same singular set of $\x$ on $V_1$. This motives us to the following definition.   
\begin{definition}
	Let $\x:U\to\R^3$ be a smooth map and $p\in U$ a singularity (point in which $\x$ is not an immersion). We say that $p$ is a {\it false singularity} if there exists open sets $V_1 \subset U$, $V_2 \subset \R^2$ with $p\in V_1$, an immersion $\y:V_2\to\R^3$ and a smooth map $\mathbf{h}:V_1\to V_2$ such that $\x=\y\circ \mathbf{h}$ on $V_1$. 
\end{definition}

\begin{proposition}
	Every smooth map $\x:U\to \R^3$ with $int(\Sigma(\x))=\emptyset$, close to a false singularity is a proper frontal with extendable normal curvature.
\end{proposition}
\begin{proof}
	If $p$ is a false singularity, then $\x=\y\circ \mathbf{h}$ on a neighborhood of $p$ with $\y$ an immersion and $\mathbf{h}$ a smooth map. Hence, $\x$ is a proper frontal with normal vector field $\n \circ \mathbf{h}$, where $\n$ is the normal vector induced by $\y$. Thus, $\Omegam=D\y(\mathbf{h})$ is a tangent moving basis of $\x$, $\Lambdam_{\Omega}^T=D\mathbf{h}$. Since, $D(\n\circ \mathbf{h})=D\n(\mathbf{h})D\mathbf{h}=D\y(\mathbf{h})\boldsymbol{\alpha}^T(\mathbf{h})D\mathbf{h}$, where $\boldsymbol{\alpha}$ is the Weingarten matrix of $\y$, we have $\x \precsim \n\circ \mathbf{h}$.
\end{proof}

\begin{proposition}
	If $\x:U\to \R^3$ is a proper wavefront, $\Omegam$ tmb of $\x$, then the normal curvature does not have a smooth extension. In particular, wavefronts have no false singularities.
\end{proposition}
\begin{proof}
	Let us suppose that the normal curvature has a smooth extension, then $\x \precsim \n$ and therefore there exists a smooth matrix valued map $\mathbf{B}$ such that $\mum_{\Omega}^T=\mathbf{B}^T\Lambdam_{\Omega}^T$, then the matrix
	$$\begin{pmatrix}
		\Lambdam_{\Omega}^T\\
		\mum_{\Omega}^T\end{pmatrix}=\begin{pmatrix}
		\Lambdam_{\Omega}^T\\
		\mathbf{B}^T\Lambdam_{\Omega}^T\end{pmatrix}=\begin{pmatrix}
		id & 0\\
		0 & \mathbf{B}^T\end{pmatrix}\begin{pmatrix}
		\Lambdam_{\Omega}^T\\
		\Lambdam_{\Omega}^T\end{pmatrix},$$
	has rank strictly less that $2$ on singularities, which is contradictory (see proposition 3.21 in \cite{med}).
\end{proof}

\section{Representation formulas of wavefronts}

In this section we obtain formulas to construct all the local parametrizations of wavefronts on a neighborhood of singularities of rank $0$ and $1$. These formulas are in terms of some functions as parameters and in most cases they can be freely chosen.  

\begin{theorem}[Representation formula for rank 1]\label{lf1}
	Let $\x:(U,0) \to (\R^3,0)$ be a germ of a wavefront, $\Omegam$ a tangent moving basis of $\x$ and $0 \in \Sigma(\x)$ with $rank(D\x(0))=1$. Then, up to an isometry $\x$ is $\mathscr{R}$-equivalent to 
\begin{align}\label{fr1}
	\y(w,z)=(w,\int_{0}^{z}\lambda_{\hat{\Omega}}(w,t)dt+f_1(w),\int_{0}^{z}t\lambda_{\hat{\Omega}}(w,t)dt+f_2(w))
\end{align}
which has as tangent moving basis   
	$$\hat{\Omegam}=\begin{pmatrix}
		&0\\
		\y_w&1\\
		&z
	\end{pmatrix}, \Lambdam_{\hat{\Omega}}=\begin{pmatrix}
		1&0\\
		0&\lambda_{\hat{\Omega}}\end{pmatrix}$$
	where $\lambda_{\hat{\Omega}}(w,z)$, $f_1(w)$, $f_2(w)$ are smooth functions with $\lambda_{\hat{\Omega}}(0)=0$. In particular, $\x$ is $\mathscr{A}$-equivalent to $(w,\int_{0}^{z}\lambda_{\hat{\Omega}}(w,t)dt,\int_{0}^{z}t\lambda_{\hat{\Omega}}(w,t)dt)$.  
\end{theorem}
\begin{proof}
	We can apply a change of coordinates $\mathbf{h}_1$ and an isometry $\boldsymbol{\phi}$ of $\R^3$ (making the line $D\x(0)(\R^2)\subset \Omegam(0)(\R^2)$ parallel to $(1,0,0)$ and the plane $\Omegam(0)(\R^2)$ coincide with $\R^2\times 0$) such that $\bar{\x}=\boldsymbol{\phi}\circ\x\circ\mathbf{h}_1=(u,b(u,v),c(u,v))$, $b_u(0,0)=b_v(0,0)=0$ and having a tangent moving basis $\Omegamb$ in the form:
	$$
	\Omegamb=\begin{pmatrix}
		1&0\\
		0&1\\
		g_1&g_2
	\end{pmatrix}$$ with $g_1(0)=g_2(0)=0$.
	Thus,  $D\bar{\x}=\Omegamb\Lambdamb^T$, $\Lambdamb^T=D(u,b)$ and $$\bar{\mum}^T=D(-g_1det(\I_{\Omegab})^{-\frac{1}{2}},-g_2det(\I_{\Omegab})^{-\frac{1}{2}}).$$ Since $\bar{\x}$ is wave front locally at $(0,0)$, by theorem \ref{wft} $$\bar{H}_{\Omegab}(0,0)=-\frac{1}{2}(-g_2det(\I_{\Omegab})^{-\frac{1}{2}})_v (0,0)\neq 0,$$ hence $g_{2v}\neq 0$. Then, by the local form of  the submersion, there exist a diffeomorphism with the form $\mathbf{h}_2(w,z)=(w,l(w,z))$ such that $g_2\circ\mathbf{h}_2=z$, therefore setting $\y(w,z):=\bar{\x}\circ\mathbf{h}_2(w,z)=(w,\tilde{b}(w,z),\tilde{c}(w,z))$, $\tilde{\Omegam}:=\Omegamb(\mathbf{h}_2)$ and $\tilde{g_1}=g_1\circ\mathbf{h}_2$ we have
	$$D\y=\Omegamb(\mathbf{h}_2)\Lambdamb^T(\mathbf{h}_2)D\mathbf{h}_2=\Omegamb(\mathbf{h}_2)D(u,b)(\mathbf{h}_2)D\mathbf{h}_2=\begin{pmatrix}
		1&0\\
		0&1\\
		\tilde{g_1}&z
	\end{pmatrix}D(w,\tilde{b})$$
	and thus $\tilde{c}_z=z\tilde{b}_z$, $\lambda_{\tilde{\Omega}}=\tilde{b}_z$. Integrating we get $\tilde{c}=\int_{0}^{z}t\lambda_{\tilde{\Omega}}(w,t)dt+\tilde{c}(w,0) $, $\tilde{b}=\int_{0}^{z}\lambda_{\tilde{\Omega}}(w,t)dt+\tilde{b}(w,0)$. Observe that the tangent moving basis $\hat{\Omegam}$ and $\Lambdam_{\hat{\Omega}}$ given in the statement of the proposition gives a decomposition of this last $\y$ in the proof, $\lambda_{\tilde{\Omega}}=\tilde{b}_z=\lambda_{\hat{\Omega}}$ and from this follows the result.  
\end{proof}
\begin{remark}[Alternative formulas for rank 1]\label{lfr}
	The formula in theorem \ref{lf1} can be rewritten in the form
	\begin{equation}\label{eq1}
		\y=(u,b(u,v),\int_{0}^{v}tb_v(u,t)dt+f_2(u)),
	\end{equation}  where $b$ is a smooth function and $b_v=\lambda_{\hat{\Omega}}$. On the other hand, observe in the proof that $D\y$ has the decomposition
$$\begin{pmatrix}
	1&0\\
	0&1\\
	g_1&v
\end{pmatrix}D(u,b),$$
where $g_1$ is a smooth function satisfying $g_{1v}=-b_u$ by corollary 3.8 in \cite{med}. Thus, on a neighborhood of $0$, there exists a smooth function $h$ such that $g_1=h_u$ and $-b=h_v$. Using this and integrating, we can get the following alternative formula:
\begin{equation}\label{eqal}
	\y=(u,-h_v(u,v),\int_{0}^{u}h_u(t,v)-h_{vu}(t,v)vdt+\int_{0}^{v}-h_{vv}(0,t)tdt),
\end{equation} 
\end{remark}

With these formulas, we can represent some sub classes with certain properties. Wavefronts with extendable Gaussian curvature and parallelly smoothable are treated in the following. We remember the notion of being parallelly smoothable, which is well determined by the behavior of the invariants near the singularities (see \cite{med2} for details).
\begin{definition}
	Let $\x:U \to \R^3$ be a wavefront and $p \in \Sigma(\x)$. We say that $\x$ is {\it parallelly smoothable at $p$} if there exist $\epsilon>0$ and an open neighborhood $V$ of $p$ such that $rank(D(\x+l\n)(q))=2$ for every $(q,l) \in V\times(0,\epsilon)$ or every $(q,l) \in V\times(-\epsilon,0)$.
\end{definition}
\begin{corollary}\label{corp1}
	Every germ of a proper wavefront $\x:(U,0) \to (\R^3,0)$ parallelly smoothable at $0$, with $rank(D\x(0))=1$, can be represented by the formula (\ref{fr1}) in which $\lambda_{\hat{\Omega}}$ is a smooth function that does not change sign on a neighborhood of $0$.  
\end{corollary}
\begin{proof}
	Using theorem \ref{lf1} and corollary 4.2 in \cite{med2} we get the result.  
\end{proof}

\begin{corollary}[Vanishing Gaussian curvature]
	Every wavefront $\x:U \to \R^3$ with  vanishing Gaussian curvature $K$, up to an isometry is $\mathscr{R}$-equivalent to the formula
	$$(u,ur_1(v)+r_2(v),\int_{0}^{v}tur'_1(t)+r'_2(t)dt+uc_1+c_2),$$
	where $r_1, r_2$ are smooth functions with $r'_2(0)=0$ and $c_1,c_2$ constants. In particular $\x$ is a ruled surface locally at $(0,0)$ with a directrix curve $(0,r_2(v),r_2(v)+c_2)$ having a singularity at $v=0$. 
	\begin{proof}
		Because $\KO(\p)\neq0$ on singularities $\p$ of rank $0$ and $\lim\limits_{(u,v)\to p}|K|=\frac{|\KO|}{|\laO|}=\infty$, then a wavefront with vanishing Gaussian curvature $\x$ only has singularities of rank $1$. Without loss of generality, let us suppose $(0,0)$ is a singularity, thus up to an isometry this is $\mathscr{R}$-equivalent to the formula in remark \ref{lfr} at $(0,0)$. Then taking the tangent moving basis in proposition \ref{lf1}, since $\Lo\No-\Muo\Mdo=0$, a simple computation leads to $-vb_{uu}+\int_{0}^{v}tb_{uuv}(u,t)dt+f_{2uu}(u)=0$. Therefore $f_{2uu}(u)=0$ and taking derivative in $v$ we get $b_{uu}=0$. From this follows the result.
	\end{proof}
\end{corollary}

\begin{corollary}[Extendable Gaussian curvature]\label{egc}
Every wavefront $\x:U \to \R^3$ with  extendable Gaussian curvature, up to an isometry is $\mathscr{R}$-equivalent to the formula	(\ref{eqal}), in which $h$ is a solution of the partial differential equation $h_{uu}+c(u,v)h_{vv}=0$, where $c(u,v)$ is a smooth function. 
\end{corollary}
\begin{proof}
	By remark \ref{lfr} we can assume $\x$ has the form of formula (\ref{eqal}). By theorem 4.2 in \cite{med2}, the Gaussian curvature is extendable if and only if $L,M,N \in \mathfrak{T}_\Omega(U)$. Using the tangent moving basis given in remark \ref{lfr}, this last results equivalent to $h_{uu}\in \mathfrak{T}_\Omega(U)$. Since $-h_{vv}=\laO$, we have that $h_{uu}+c(u,v)h_{vv}=0$ on $U$ for some smooth function $c(u,v)$. 
\end{proof}
\begin{remark}
	Depending on the chosen function $c(u,v)$, we could get explicit solutions $h$ of the equation $h_{uu}+c(u,v)h_{vv}=0$. For example if we take $c<0$ a negative constant, the last equation is the well known Wave equation which has as general solution $h(u,v)=h_1(v-\sqrt{-c}u)+h_2(v+\sqrt{-c}u)$, where $h_1, h_2$ are smooth functions. Also, choosing $c>0$ a positive constant, after a change of coordinates this equation is the well known Laplace equation which leads to the general solution $h(u,v)=F(u,v/\sqrt{c})$, where $F$ is a harmonic function. 
\end{remark}
Now, we will get representation formulas for wavefronts near singularities of rank 0. These type of singularities are less common than the ones of rank 1 and the geometrical invariants have a behavior totally different (see \cite{med2}).
\begin{proposition}\label{spf}
	Let $\x:(U,0) \to (\R^3,0)$ be a germ of a wavefront, $\Omegam$ a tangent moving basis of $\x$ and $0 \in \Sigma(\x)$ with $\KO(0)\neq 0$. Then, up to an isometry $\x$ is $\mathscr{R}$-equivalent to $\y=(a,b,\int_{0}^{u}t_1a_u(t_1,v)+vb_u(t_1,v)dt_1+\int_{0}^{v}t_2b_v(0,t_2)dt_2)$, where $a$, $b$ are smooth functions and $a_v=b_u$. 
\end{proposition}
\begin{proof}
	Applying an isometry we always can choose a tangent moving basis of $\x=(a,b,c)$ in the form 
	$$
	\Omegam=\begin{pmatrix}
		1&0\\
		0&1\\
		g_1&g_2
	\end{pmatrix}, \Lambdam_\Omega^T=D(a,b).$$
	We have that $\KO(0)\neq 0$ if and only if $det(\II_{\Omega}(0))\neq0$ and by a simple computation this is equivalent to have $det(D(g_1,g_2)(0))\neq0$, therefore by the Inverse function theorem there exists a diffeomorphism $\mathbf{h}(w,z)$ on a small neighborhood such that $(g_1,g_2)\circ\mathbf{h}=(w,z)$. Setting $\y:=\x\circ\mathbf{h}=(\hat{a},\hat{b},\hat{c})$ we have
	$$D\y=\Omegam(\mathbf{h})\Lambdam_\Omega^T(\mathbf{h})D\mathbf{h}=\begin{pmatrix}
		1&0\\
		0&1\\
		w&z
	\end{pmatrix}D(\hat{a},\hat{b}),$$ 
	thus by corollary 3.8 in \cite{med} $(\hat{a},\hat{b})_w\cdot(w,z)_z=(\hat{a},\hat{b})_z\cdot(w,z)_w$, it means $\hat{b}_w=\hat{a}_z$. Also, $\hat{c}_w=w\hat{a}_w+z\hat{b}_w$ and $\hat{c}_z=w\hat{a}_z+z\hat{b}_z$. Then, $\hat{c}=\int_{0}^{w}t_1\hat{a}_w(t_1,z)+z\hat{a}_z(t_1,z)dt_1+\hat{c}(0,z)$, but $\hat{c}(0,z)=\int_{0}^{z}t_2\hat{b}_z(0,t_2)dt_2$, from this follows the result. 
\end{proof}

\begin{remark}
	Observe that the condition $a_v=b_u$ near $0$ is equivalent to the existence of a smooth function $h$ such that $a=h_u$ and $b=h_v$ on some neighborhood of $0$.
\end{remark}

\begin{theorem}[Representation formula for rank 0]\label{fr0}
	Let $\x:(U,0) \to (\R^3,0)$ be a germ of a wavefront, $\Omegam$ a tangent moving basis of $\x$ and $0 \in \Sigma(\x)$ with $rank(D\x(0))=0$. Then, up to an isometry $\x$ is $\mathscr{R}$-equivalent to 
	\begin{align}\label{fr2}
		\y=(h_u,h_v,\int_{0}^{u}t_1h_{uu}(t_1,v)+vh_{vu}(t_1,v)dt_1+\int_{0}^{v}t_2h_{vv}(0,t_2)dt_2),
	\end{align}
where $h$ is a smooth function defined on some neighborhood of $(0,0)$ with $h_{uu}(0)=h_{uv}(0)=h_{vv}(0)=0$.
\end{theorem}

\begin{proof}
	By proposition \ref{wft} $\KO(0)\neq0$ and applying the proposition \ref{spf} we get the result.
\end{proof}

\begin{corollary}
	Let $\x:(U,0) \to (\R^3,0)$ be a germ of a wavefront, $\Omegam$ a tangent moving basis of $\x$ and $0 \in \Sigma(\x)$. Then, $\x$ is $\mathscr{A}$-equivalent to $\y=(h_u,h_v,\int_{0}^{u}t_1h_{uu}(t_1,v)+vh_{vu}(t_1,v)dt_1+\int_{0}^{v}t_2h_{vv}(0,t_2)dt_2)$, where $h$ is a smooth function defined on some neighborhood of $(0,0)$. 
\end{corollary}
\begin{proof}
	The case $rank(D\x(0))=0$ is the last corollary. If $rank(D\x(0))=1$, by proposition \ref{lf1} $\x$ is $\mathscr{A}$-equivalent to $(w,\int_{0}^{z}\lambda_{\hat{\Omega}}(w,t)dt,\int_{0}^{z}t\lambda_{\hat{\Omega}}(w,t)dt)$ which is $\mathscr{A}$-equivalent to $(w,\int_{0}^{z}\lambda_{\hat{\Omega}}(w,t)dt,\int_{0}^{z}t\lambda_{\hat{\Omega}}(w,t)dt+w^2)$. By a simple computation for this last wavefront $K_{\hat{\Omega}}(0)\neq0$ and applying proposition \ref{spf} we get the result.    
\end{proof}

\begin{remark}
	The proof of this corollary give us an algorithm to transform a wavefront with singularities of rank 1 into a wavefront with non-vanishing $\KO$ using change of coordinates at the target.
\end{remark}

\begin{corollary}\label{corp0}
	Every wavefront $\x:(U,0) \to (\R^3,0)$ parallelly smoothable at $0$ with $rank(D\x(0))=0$ can be represented with the formula (\ref{fr2}), in which $h$ is a smooth concave or convex function on a convex neighborhood of $0$ with $h_{uu}(0)=h_{uv}(0)=h_{vv}(0)=0$. 
\end{corollary}
\begin{proof}
By theorem \ref{fr0} we can assume that $\x$ has the form of formula (\ref{fr2})	which has as tangent moving basis
$$
\Omegam=\begin{pmatrix}
	1&0\\
	0&1\\
	u&v
\end{pmatrix}, \Lambdam_\Omega^T=\begin{pmatrix}
h_{uu}&h_{uv}\\
h_{uv}&h_{vv}
\end{pmatrix}.$$
By theorem 5.1 in \cite{med2} $\x$ is parallelly smoothable at $0$ if and only if $\laO \KO\geq 0$ and $\HO$ does not change sign on a neighborhood of $0$. With a simple computation of $\laO \KO$ and $\HO$ this last is equivalent to have $h_{uu}h_{vv}-h_{uv}^2\geq 0$ and $h_{uu}+h_{vv}+u^2h_{uu}+2uvh_{uv}+v^2h_{vv}$ does not change sign on a neighborhood of $0$. Since $h_{uu}h_{vv}-h_{uv}^2\geq 0$, then $h_{uu}+h_{vv}\geq 0$ (resp. $\leq 0$) is equivalent to $\Lambdam_\Omega^T$ being positive semi-definite (resp. negative semi-definite), therefore $h_{uu}+h_{vv}+u^2h_{uu}+2uvh_{uv}+v^2h_{vv}$ does not change sign if and only if $h_{uu}+h_{vv}$ neither. Thus, $\x$ is parallelly smoothable at $0$ if and only if $\Lambdam_\Omega^T$ is positive or negative semi-definite. This last is equivalent to $h$ being a convex or concave function on a convex neighborhood of $0$ (see \cite{B}).     
\end{proof}

\section{Some applications to asymptotic curves and lines of curvature}
In this section, we apply some of the results previously obtained to describe how are the asymptotic curves and lines of curvatures distributed. These curves result kind of similar to the regular case, but with the possibility of having singularities.
\begin{definition}
	Let $\x:U\to \R^3$ be a proper frontal and $\gamma:(-\epsilon,\epsilon)\to U$ a smooth curve. We say that $\gamma$ is an asymptotic curve of $\x$ or simply $\gamma$ is asymptotic if $(\x\circ \gamma)'(t)$ defines an asymptotic direction for every $t$ such that $(\x\circ \gamma)'(t)\neq0$.
\end{definition}
Let $\x:U\to \R^3$ a proper frontal, $\Omegam$ a tmb of $\x$. It easy to see that $\gamma$ is an asymptotic curve of $\x$ if and only if $\gamma'^T\Lambdam_{\Omega}(\gamma)\II_{\Omega}(\gamma)adj(\Lambdam_{\Omega}^T(\gamma))\Lambdam_{\Omega}^T(\gamma)\gamma'=0$ on $(-\epsilon,\epsilon)$, which is the same as $\laO(\gamma)\gamma'^T\II(\gamma)\gamma'=0$ due to the fact that $\II=\Lambdam_{\Omega}\II_{\Omega}$. From this, we can conclude that every curve $\gamma$ contained  in $\Sigma(\x)$ is asymptotic. 
\begin{definition}
	Let $\x:U\to \R^3$ be a frontal and $\gamma(t):(-\epsilon,\epsilon)\to U$ a smooth curve. We say that $\gamma$ is a Gaussian asymptotic (G-asymptotic) curve of $\x$ or simply $\gamma$ is G-asymptotic if $\gamma'(t)^T\II\gamma'(t)=0$ for all $t\in (-\epsilon, \epsilon)$, where the entries of $\II$ are being evaluated at $\gamma(t)$.
\end{definition}
It is immediate that G-asymptotic curves are asymptotic. The following theorems are in terms of these curves and describe how are organized asymptotic curves on frontals with extendable normal curvature and wavefronts with extendable Gaussian curvature.
\begin{theorem}\label{asyt1}
	Let $\x:U\to \R^3$ a proper frontal, $\Omegam$ a tmb of $\x$ and $p\in U$ a singularity. If the normal curvature has a smooth extension and the extension of the Gaussian curvature is strictly negative, then there exists two families of curves (not necessary regular) $\phi_1(t,q):J\times U_1\to U_2$, $\phi_2(t,q):J\times U_1\to U_2$, where $J$ is an open interval containing $0$, $U_1, U_2$ are open sets of $\R^2$ with $p\in U_2$, $U_1\subset U_2\subset U$, satisfying the following:
	\begin{enumerate}[label=(\roman*)]
		\item $\phi_1$ and $\phi_2$ are smooth and $\phi_1(0,q)=\phi_2(0,q)=q$. 
		
		\item For each fixed $q$, $\phi_1(t,q)$ and $\phi_2(t,q)$ are G-asymptotic curves of $\x$. 
		
		\item For every $t_0 \in J$ such that $\phi_i(t_0,q)\in \Sigma(\x)^c$, we have $\phi_i'(t_0,q)\neq0$. In the case that $\phi_i(t_0,q)\in \Sigma(\x)^c$ for $i=1,2$ simultaneously, then $\phi_1'(t_0,q)$ and $\phi_2'(t_0,q)$ are linearly independent.  
	\end{enumerate}
\end{theorem}
\begin{proof}
	By corollary \ref{lamb} there exist a smooth matrix-valued map $\mathbf{B}:U \to GL(2,\R)$, such that $\mum_{\Omega}=\Lambdam_{\Omega}\mathbf{B}$ and by proposition \ref{lim} $K=det(\mathbf{B})$. Since that $\boldsymbol{\mu}_{\Omega}:=-\II_\Omega^T\I_\Omega^{-1}$, then $\II_{\Omega}=\mathbf{C}\Lambdam_{\Omega}^T$ where $\mathbf{C}=(c_{ij})$ is a smooth matrix-valued map with $det(\mathbf{C})<0$ on $U$. As $\II=\Lambdam_{\Omega}\II_{\Omega}=\Lambdam_{\Omega}\mathbf{C}\Lambdam_{\Omega}^T$ and $\II$ is symmetric, we have that $\mathbf{C}$ is symmetric. Without loss of generality, we can suppose that $p=0$. Denoting $\rho(t)=(\rho_1(t),\rho_2(t))=\Lambdam_{\Omega}^T\gamma'(t)$ for a smooth curve $\gamma$ with $\gamma(0)=0$, where $\Lambdam_{\Omega}$ is being evaluated in $\gamma(t)$, we have that $$\gamma'(t)^T\II\gamma'(t)=\rho(t)^T \mathbf{C}\rho(t)=c_{11}\rho_1^2+2c_{12}\rho_1\rho_2+c_{22}\rho_2^2.$$
	Since that $c_{11}c_{22}-c_{12}^2<0$, the last equation can be decomposed into linear factors, yielding $$\gamma'(t)^T\II\gamma'(t)=(a_1\rho_1+a_2\rho_2)(a_1\rho_1+a_3\rho_2),$$
	where $a_1,a_2,a_3$ are smooth real function on a neighborhood of $0\in U$, such that $a_1^2=c_1, a_2a_3=c_{22},a_1(a_2+a_3)=2c_{12}$. Shrinking $U$ if is necessary, we have just two cases:\\
	Case 1. $c_{12}\neq0$ on $U$. This implies that, the vector fields $(-a_2,a_1)$ and $(-a_3,a_1)$ are linearly independent. Observe that curves $\gamma$ satisfying the differential equations 
\begin{align}\label{sys1}
	\gamma_1'=adj(\Lambdam_{\Omega})^T\begin{pmatrix}
		-a_2\\
		 a_1
	\end{pmatrix}, 
\end{align} 
\begin{align}\label{sys2}
	\gamma_2'=adj(\Lambdam_{\Omega})^T\begin{pmatrix}
		-a_3\\
		 a_1
	\end{pmatrix}, 
\end{align}
are G-asymptotic. If we take $\phi_1(t,q)$, $\phi_2(t,q)$ as the local flows at $0$ of (\ref{sys1}) and (\ref{sys2}), we get the result.\\
Case 2. $c_{11}(0,0)c_{22}(0,0)\neq0$ and $c_{12}(0,0)=0$. Here the vector fields $(a_2,-a_1)$ and $(-a_3,a_1)$ are linearly independent on a neighborhood of $0$ and analogously to the case 1 we get the result.    
\end{proof}

\begin{theorem}\label{asyt2}
	Let $\x:U\to \R^3$ a proper wavefront, $\Omegam$ a tmb of $\x$. If the Gaussian curvature is extendable and the extension is strictly negative, then locally at a singularity $p$, there exist open sets $U_1, U_2$ of $\R^2$ with $p\in U_2$, $U_2\subset U$ and a diffeomorphism $\mathbf{t}:U_1\to U_2$ such that the coordinated curves are G-asymptotic curves of $\x$. 
\end{theorem}
\begin{proof}
	By corollary \ref{egc}, we can assume that $\x$ is equal to the formula (\ref{eqal}) with $h_{uu}+c(u,v)h_{vv}=0$ for some smooth function $c(u,v)$ and $p=0\in U$ is a singularity. Computing the Gaussian curvature with this formula results $K(u,v)=c(u,v)(1+h_u^2+v^2)^{-2}$, then $c(u,v)<0$ on $U$. We also have, 
	$$\II=\begin{pmatrix}
		1&-h_{vu}\\
		0&-h_{vv}
	\end{pmatrix}\begin{pmatrix}
		h_{uu}&h_{uv}\\
		0&1
	\end{pmatrix}(1+h_u^2+v^2)^{-\frac{1}{2}}.$$
Therefore, $\gamma'(t)^T\II\gamma'(t)=0$ if and only if $c(\gamma(t)) h_{vv}(\gamma(t)) u'(t)^2+h_{vv}(\gamma(t)) v'(t)^2=0$. Observe that curves $\gamma$ with derivative in the direction of the vector fields
\begin{align}\label{asym}
(-1,(-c(\gamma(t)))^{\frac{1}{2}}) \text{ or } (1,(-c(\gamma(t)))^{\frac{1}{2}})\end{align} are G-asymptotic. As the vector fields $(-1,(-c(u,v))^{\frac{1}{2}})$, $(1,(-c(u,v))^{\frac{1}{2}})$ are linearly independent on $U$, it possible to find open sets $U_1, U_2$ of $\R^2$ with $p\in U_2$, $U_2\subset U$ and a diffeomorphism $\mathbf{t}:U_1\to U_2$ such that the coordinated curves are tangent to the lines in the directions of (\ref{asym})(see 3-4 in \cite{dc}). It follows the result.     
\end{proof}
In the proof of the above theorem, curves $\gamma$ satisfying $h_{vv}(\gamma(t))=0$ are G-asymptotic as well as these are the curves contained in the singular set of $\x$. This is in fact true in a more general context. It is valid for bounded Gaussian curvature and does not depend on the chosen coordinates as we will see in the next proposition.
\begin{proposition}
	If $\x:U\to \R^3$ is a wavefront, $\Omegam$ a tmb of $\x$ and the Gaussian curvature is bounded, then every curve $\gamma:(-\epsilon,\epsilon)\to U$ with $\gamma((-\epsilon,\epsilon))\subset \Sigma(\x)$ is a G-asymptotic curve.
\end{proposition}
\begin{proof}
By theorem 4.1 in \cite{med2}, there exists a constant $C>0$ such that $|L|\leq C|\laO|$, $|M|\leq C|\laO|$ and $|N|\leq C|\laO|$. Denoting $\gamma(t)=(u(t),v(t))$, we have
$$|Lu'^2+2Mu'v'+Nv'^2|\leq C|\laO(u(t),v(t))|(u'^2+2|u'v'|+v'^2)=0$$
and it follows the result.	
\end{proof}

In the following, we give a similar result to the theorem \ref{asyt1} about lines of curvature for frontals with extendable normal curvature. In the case of wavefronts, near singularities of rank 1, it is known that there exist a diffeomorphism like in the theorem \ref{asyt2} for lines of curvatures (see \cite{murataumehara} for example).
\begin{definition}
	Let $\x:U\to \R^3$ be a proper frontal and $\gamma:(-\epsilon,\epsilon)\to U$ a smooth curve. We say that $\gamma$ is a line of curvature of $\x$ if $(\x\circ \gamma)'(t)$ defines a principal direction for every $t$ such that $(\x\circ \gamma)'(t)\neq0$.
\end{definition} 
\begin{remark}
	Observe that, $(\x\circ \gamma)'(t)\neq0$ if and only if $\Lambdam_{\Omega}^T(\gamma(t))\gamma'(t)\neq0$, where this last vector is the coordinate of $(\x\circ \gamma)'(t)$ in the basis $\Omegam$, then by remark \ref{eigenv} we have that $(\x\circ \gamma)'(t)$ defines a principal direction if and only if $\Lambdam_{\Omega}^T(\gamma(t))\gamma'(t)$ is an eigenvector of $-\mum_{\Omega}^Tadj(\Lambdam_{\Omega}^T)$ evaluated at $\gamma(t)$. Thus, $(\x\circ \gamma)'(t)$ defines a principal direction if and only if there exists a scalar $l(t)$ such that $\laO(\gamma'(t))\mum_{\Omega}^T(\gamma(t))\gamma'(t)=l(t)\Lambdam_{\Omega}^T(\gamma(t))\gamma'(t)$ which is equivalent to $\laO(\gamma(t))(\n \circ \gamma)'(t)=l(t)(\x \circ \gamma)'(t)$. As $(\x\circ \gamma)'(t_0)=0$ implies that $\laO(\gamma(t_0))=0$, we can extend the last equality to the entire interval $(-\epsilon,\epsilon)$ simply defining $l(t_0)$ with an arbitrary value and finally we get the following proposition.       
\end{remark}
\begin{proposition}\label{princi}
	Let $\x:U\to \R^3$ be a proper frontal and $\Omegam$ a tmb of $\x$. A smooth curve $\gamma:(-\epsilon,\epsilon)\to U$ is a line of curvature if and only if there exist a function $l(t):(-\epsilon,\epsilon)\to \R$ such that $\laO(\gamma(t))(\n\circ \gamma)'(t)=l(t)(\x \circ \gamma)'(t)$.
\end{proposition}
From the last proposition, choosing $l(t)=\laO(\gamma(t))$ we deduce immediately that smooth curves $\gamma$ contained in the singular set $\Sigma(\x)$ are lines of curvature. The function $l(t)$ in proposition \ref{princi} may not be continuous, but we can characterize line of curvatures with a equation that does not involve $l(t)$ as follows.
\begin{corollary}
	Let $\x:U\to \R^3$ be a proper frontal and $\Omegam$ a tmb of $\x$. A smooth curve $\gamma:(-\epsilon,\epsilon)\to U$ is a line of curvature if and only if $\laO(\gamma)\gamma'^T\mathbf{P}\boldsymbol{\alpha}_{\Omega}^T(\gamma)\gamma'=0$ on $(-\epsilon,\epsilon)$, where $$\mathbf{P}=\begin{pmatrix}
		0&1\\
		-1&0
	\end{pmatrix}.$$
\end{corollary}
\begin{proof}
	If $\gamma$ is a line of curvature, then by proposition \ref{princi}, there exists a function $l(t):(-\epsilon,\epsilon)\to \R$ such that $\laO(\gamma)\mum_{\Omega}^T(\gamma)\gamma'=l(t)\Lambdam_{\Omega}^T(\gamma)\gamma'$ which implies $\laO(\gamma)\boldsymbol{\alpha}_{\Omega}^T(\gamma)\gamma'=\laO(\gamma)l(t)\gamma'$. As $\gamma'^T\mathbf{P}\gamma'=0$ then we get $\laO(\gamma)\gamma'^T\mathbf{P}\boldsymbol{\alpha}_{\Omega}^T(\gamma)\gamma'=0$ on $(-\epsilon,\epsilon)$. For the converse, if $t\in (-\epsilon,\epsilon)$ is such that $\laO(\gamma(t))=0$ or $\gamma'(t)=0$ let us define $l(t):=0$ and if $t\in(-\epsilon,\epsilon)$ is such that $\laO(\gamma(t))\neq0$ with $\gamma'(t)\neq0$, we have that $\gamma'(t)^T\mathbf{P}\boldsymbol{\alpha}_{\Omega}^T(\gamma(t))\gamma'(t)=0$, namely $\mathbf{P}^T\gamma'(t)$ and $\boldsymbol{\alpha}_{\Omega}^T(\gamma(t))\gamma'(t)$ are orthogonal, then there exists a scalar $r(t)$ such that $\boldsymbol{\alpha}_{\Omega}^T(\gamma(t))\gamma'(t)=r(t)\gamma'(t)$ and hence $\laO(\gamma)\mum_{\Omega}^T(\gamma)\gamma'=r(t)\Lambdam_{\Omega}^T(\gamma)\gamma'$.Thus, if we define $l(t):=r(t)$ for those $t$, we have that $\laO(\gamma(t))(\n\circ \gamma)'(t)=l(t)(\x \circ \gamma)'(t)$ for every $t\in (-\epsilon,\epsilon)$.     
\end{proof} 
\begin{definition}
Let $\x:U\to \R^3$ be a proper frontal and $\gamma:(-\epsilon,\epsilon)\to U$ a smooth curve. We say that $\gamma$ is a Gaussian line of curvature of $\x$ if there exists a function $l(t):(-\epsilon,\epsilon)\to \R$ such that $(\n \circ \gamma)'(t)=l(t)(\x \circ \gamma)'(t)$.
\end{definition}
We can see easily that Gaussian lines of curvatures are lines of curvature. Now, we prove our last result.
\begin{theorem}\label{lct1}
	Let $\x:U\to \R^3$ a proper frontal, $\Omegam$ a tmb of $\x$. If the normal curvature has a smooth extension and the extension of the principal curvatures are different at a singularity $p$, then there exists two families of curves (not necessary regular) $\phi_1(t,q):J\times U_1\to U_2$, $\phi_2(t,q):J\times U_1\to U_2$, where $J$ is an open interval containing $0$, $U_1, U_2$ are open sets of $\R^2$ with $p\in U_2$, $U_1\subset U_2\subset U$, satisfying the following:
	\begin{enumerate}[label=(\roman*)]
		\item $\phi_1$ and $\phi_2$ are smooth and $\phi_1(0,q)=\phi_2(0,q)=q$. 
		
		\item For each fixed $q$, $\phi_1(t,q)$ and $\phi_2(t,q)$ are Gaussian lines of curvature of $\x$. 
		
		\item For every $t_0 \in J$ such that $\phi_i(t_0,q)\in \Sigma(\x)^c$, we have $\phi_i'(t_0,q)\neq0$. In the case that $\phi_i(t_0,q)\in \Sigma(\x)^c$ for $i=1,2$ simultaneously, then $\phi_1'(t_0,q)$ and $\phi_2'(t_0,q)$ are linearly independent.  
	\end{enumerate}
\end{theorem}
\begin{proof}
	By theorem \ref{tp1} $\x \precsim \n$, then there exist a smooth matrix-valued map $\mathbf{B}:U \to \mathcal{M}_{2\times2}(\R)$, such that $\mum_{\Omega}=\Lambdam_{\Omega}\mathbf{B}$, therefore $\boldsymbol{\alpha}_{\Omega}=\boldsymbol{\mu}_{\Omega}adj(\Lambdam_\Omega)=\Lambdam_{\Omega}\mathbf{B}adj(\Lambdam_\Omega)$. Thus, $K=det(\mathbf{B})$ and $H=-\frac{1}{2}tr(\mathbf{B})$. As the eigenvalues of $\mathbf{B}^T$ are $$\rho_{\pm}=\frac{1}{2}(tr(\mathbf{B}^T))\pm \sqrt{(\frac{1}{2}(tr(\mathbf{B}^T)))^2-det(\mathbf{B}^T)},$$ we have that $\rho_{-}$ and $\rho_{+}$ are real and different on a neighborhood of $p$. Therefore, the ranks of $\mathbf{B}^T-\rho_{+}id$, $\mathbf{B}^T-\rho_{-}id$ are 1 and shrinking $U$ if it is necessary, we can find vector fields $\eta_1$, $\eta_2$ linearly independent, being in the kernel of these last matrices respectively. Now, let $\phi_1(t,q)$, $\phi_2(t,q)$ be the local flows at $p\in U$ of the differential equations $\gamma'=adj(\Lambdam_{\Omega})^T\eta_1(\gamma)$, $\gamma'=adj(\Lambdam_{\Omega})^T\eta_2(\gamma)$ respectively. We have that \begin{align*}
	&(\n \circ \phi_i)'(t)=\Omegam\mum_{\Omega}^T\phi_i'(t)=\Omegam\mathbf{B}^T\Lambdam_{\Omega}^Tadj(\Lambdam_{\Omega})^T\eta_i=\rho_{\pm}\Omegam\laO(\phi_i(t))\eta_i=\\
	&\rho_{\pm}\Omegam\Lambdam_{\Omega}^Tadj(\Lambdam_{\Omega})^T\eta_i=\rho_{\pm}D\x\phi_i'(t)=\rho_{\pm}(\x \circ \phi_i)'(t),	
	\end{align*} 
where all the functions and matrix-valued maps are being evaluated at $\phi_i(t)$. It follows the result.
\end{proof}
\section{Declarations}
{\bfseries Competing interests:} On behalf of all authors, the corresponding author states that there is no conflict of interest.

{\bfseries Availability of data:} Data sharing not applicable to this article as no datasets were generated or analysed during the current study.

\def\cprime{$'$}

\end{document}